\documentclass[12pt,reqno]{amsart}
\usepackage{color}
\usepackage{graphicx}
\usepackage{amssymb}
\usepackage{amsthm}
\usepackage{amsmath}
\usepackage{graphics}
\usepackage{csquotes}
\usepackage{epstopdf}
\usepackage{color}
\usepackage{tikz}
\usepackage{cite}
\newtheorem{theorem}{Theorem}[section]
\newtheorem{lemma}[theorem]{Lemma}

\newtheorem{definition}[theorem]{Definition}
\newtheorem{proposition}[theorem]{Proposition}
\newtheorem{obs}{Remark}[section]

\begin{document}

\title[Generalized Fractional Benjamin-Bona-Mahony Equation]{On the Stability of Solitary Wave Solutions for a Generalized Fractional Benjamin-Bona-Mahony Equation }


\author{Goksu Oruc}

\address[G. Oruc]{Istanbul Technical University, Department of Mathematics, Maslak,
	Istanbul,  Turkey.}
\email{topkarci@itu.edu.tr }

\author{F\'abio Natali}

\address[F. Natali]{Department of Mathematics, State University of
	Maring\'a, Maring\'a, PR, Brazil. }
\email{fmanatali@uem.br }

\author{Handan Borluk}

\address[H. Borluk]{Ozyegin University, Department of Natural and Mathematical Sciences, Cekmekoy, Istanbul, Turkey. }
\email{handan.borluk@ozyegin.edu.tr }

\author{Gulcin M. Muslu}

\address[G. M. Muslu]{Istanbul Technical University, Department of Mathematics, Maslak,
	Istanbul,  Turkey.}
\email{gulcin@itu.edu.tr }

\maketitle

\begin{abstract}
In this paper we establish a rigorous spectral stability analysis for solitary waves associated to a generalized fractional Benjamin-Bona-Mahony type equation. Besides the well known smooth and positive solitary wave with large wave speed, we present the existence of smooth negative solitary waves having small wave speed. The spectral stability is then determined by analysing the behaviour of the associated linearized operator around the wave restricted to the orthogonal of the tangent space related to the momentum at the solitary wave. Since the analytical solution is not known, we generate the negative solitary waves numerically by using Petviashvili method. We also present some numerical experiments to observe the stability properties of solitary waves for various  values of the order of nonlinearity and fractional derivative. Some remarks concerning the orbital stability are also celebrated.
\end{abstract}

\renewcommand{\theequation}{\arabic{section}.\arabic{equation}}
\setcounter{equation}{0}
\section{Introduction}

In this paper we consider the stability properties of the solitary wave solutions for the generalized fractional Benjamin-Bona-Mahony (gfBBM) equation  given by
\begin{equation}
    u_t+  u_x + \frac{1}{2}(u^{p+1})_x+   \frac{3}{4}D^{\alpha} u_{x}+ \frac{5}{4}D^{\alpha} u_{t}=0. \label{gfBBM}
\end{equation}
Here $p$ is an integer and the operator $D^{\alpha} =(-
\partial_x^2)^{\frac{\alpha}{2}}$ denotes the Riesz potential of order $-\alpha$,
for any $\alpha \in \mathbb{R}$. The operator $D^{\alpha}$ is defined via
Fourier transform by
\begin{equation*}
  \widehat{D^{\alpha}q}(\xi)=|\xi|^{\alpha}\hat{q}(\xi),
\end{equation*}
where $\hat{q}$ is the Fourier transform in $L^2(\mathbb{R})$ of the function $q$.
The gfBBM equation admits the following conserved quantities
\begin{eqnarray}
  I(u)&=& \int_{-\infty}^{+\infty} u dx, \label{cq1} \\
  F(u)&=& \frac{1}{2} \int_{-\infty}^{+\infty} \left(u^2 + \frac{5}{4}|D^{\frac{\alpha}{2}}u |^2\right) dx, \label{cq2}\\
  H(u)&=&-\frac{1}{2}\int_{-\infty}^{\infty} \left({u^2}+ \frac{u^{p+2}}{p+2}+ \frac{3}{4}|D^{\frac{\alpha}{2}}u|^2 \right)dx.   \label{hamilton}
\end{eqnarray}
\noindent Equation \eqref{gfBBM} is derived to model the propagation of small amplitude long unidirectional waves in a non-locally and non-linearly elastic medium  \cite{EEE1,EEE2}.\\
\indent The Korteweg-de Vries (KdV) equation
\begin{equation*}
   u_t+  u_x + u u_x+u_{xxx}=0
\end{equation*}
is one of the most celebrated equations in  water wave theory.  Due to the shortcomings of the KdV equation, new approximate models are developed for water waves.  A well-known alternative model is the Benjamin-Bona-Mahony (BBM) equation
\begin{equation*}
   u_t+  u_x + u u_x-u_{xxt}=0
\end{equation*}
derived in \cite{benjamin2}. The dispersion relation for the linear BBM equation is the Pad\'{e} (0,2)-approximation to the full dispersion relation. To provide a better approximation of the dispersion relation for the small wave numbers, the Pad\'{e} (2,2)-approximation is used  to derive the equation,
\begin{equation}\label{pade}
  u_t+u_x+uu_x-\frac{9}{10}u_{xxx}-\frac{19}{10}u_{xxt}=0,
\end{equation}
which involves both linear terms for the KdV and BBM equations \cite{levy}.
The same equation is also established for Boussinesq-like equations with improved frequency dispersion   to model unidirectional surface water waves \cite{dingemans}. The authors in \cite{kalisch3} showed numerically that the solutions of the equation  \eqref{pade} provides a  better approximation to the solutions of the full Euler equations than either the KdV or BBM equations.

The effects of the relation between the nonlinearity and the dispersion
on the dynamics of solutions has been the focus of many studies. For this aim,
the fractional equations are  widely studied  in recent years  \cite{linares1,linares2,klein,pava,duran,natali, le}.
The fractional forms of the KdV equation and  the BBM equation are given as the fractional KdV (fKdV) equation
\begin{equation}\label{fkdv}
   u_t+  u_x + u^p u_x- D^{\alpha}u_x=0,
\end{equation}
 and  fractional BBM (fBBM) equation
 \begin{equation}\label{fbbm}
   u_t+  u_x + u^p u_x+ D^{\alpha}u_t=0,
\end{equation}
respectively.
The gfBBM equation involves the fractional terms of both the fKdV and fBBM equations. The efficiency of the equation \eqref{pade} makes the gfBBM equation more interesting. This paper gives an insight into the dynamics of the solutions for the gfBBM equation.

The local well-posedness of Cauchy problem for both equations $(\ref{fkdv})$ and $(\ref{fbbm})$ is investigated in \cite{abdelouhab} and \cite{linares1}. In the case $p=1$, the missing uniqueness result is recovered  for the equation $(\ref{fbbm})$ in \cite{he}. 
The existence and orbital stability of positive solitary wave solutions with the form $u(x,t)=Q_c(x-ct)$ for both fKdV and fBBM equations are discussed in \cite{linares2} and \cite{pava} using variational methods. The case $\alpha=2$ deserves to be highlighted as it corresponds to the classical KdV and BBM equations. Concerning the equation $(\ref{fkdv})$, the authors in \cite{bona} have exhibited a scenario for the orbital stability of solitary waves with {\rm sech}-type profile for $c>0$. In fact, orbital stability in the energy space $H^1(\mathbb{R})$ occurs for $1\leq p<4$ and orbital instability is determined for $p>4$. The case $p=4$ is well known as $L^2-$critical case and orbital instability results have been established in \cite{LG} and \cite{martel}. For the equation $(\ref{fbbm})$ and $c>1$, we have a similar scenario. Indeed, using suitable adaptations of the approach in \cite{bona}, the authors in \cite{strauss} have established the orbital stability of solitary waves with $sech-$profile for the case $1\leq p\leq4$. For the case $p>4$, the solitary waves are orbitally unstable for $1<c<c_0$ and orbitally stable for $c>c_0$. Here, $c_0$ is a critical value depending on $p$.

\indent The existence, uniqueness and spectral stability of periodic waves for the
fKdV and the fBBM equations  are also studied in \cite{le} and \cite{natali}.
The stability behaviour of the solutions changes drastically for the fKdV and the fBBM equations. To illustrate in the case $p=1$, it has been shown that  the solitary wave solutions of the fKdV equation are orbitally stable for $c>0$ when $\alpha \in (\frac{1}{2}, 2)$ and spectrally unstable when $\alpha \in (\frac{1}{3}, \frac{1}{2})$.  The solitary wave solutions of the fBBM equation are orbitally stable when $c>1$ and $\alpha \in [\frac{1}{2}, 2)$ or when $c>c_1>1$ and $\alpha \in (\frac{1}{3}, \frac{1}{2})$ ($c_1$ is also a critical value depending on $\alpha$). The solutions are spectrally unstable when $\alpha \in (\frac{1}{3}, \frac{1}{2})$ and $1<c<c_1$.
Therefore the stability problem for the solitary wave solutions of the gfBBM equation is crucial since the equation involves the  fractional terms of both the  fKdV and the fBBM equations.

An important feature of the gfBBM equation is  that it admits negative solitary wave solutions. In this study, we prove  the existence of negative solitary wave solutions for small wave speed and odd values of the order of nonlinearity $p$. To the best of our knowledge, the existence of negative solitary wave  has never been investigated for the fKdV and fBBM equations when $\alpha \in (0,2)$. The negative solitary waves first appeared in \cite{lewis} for a regularized long wave model in water waves. For the BBM equation, $\alpha=2$ in the equation $(\ref{fbbm})$, the authors in \cite{kalisch} established the existence of negative solitary waves when $c<0$ and $p$ odd. Using arguments of symmetry and the existence of positive solutions $Q_c$ with $c>1$ (see \cite{strauss}), it is possible to prove that the case $p$ even also admits negative solitary waves of the form $-Q_c$ for $c>1$. The orbital stability of negative solitary waves has been also treated in \cite{kalisch} (see \cite{kalisch1} for a complementary result). For $p\geq1$ odd and $c<0$, there exists a critical value $c_2<0$ such that the solitary wave is orbitally unstable for $c_2<c<0$ and orbitally stable for $c<c_2$. For $p>2$ even, we have a similar scenario as in the case of positive and solitary waves since the negative solution in this case is given by $-Q_c$, where $Q_c$ is the positive solution for $c>1$.

\indent One of the main properties which guarantee the existence of solitary waves of the form $u(x,t)=Q_c(x-ct)$ is the presence of translation symmetry at the spatial variable associated to the equation $(\ref{gfBBM})$. In fact, if $u(x,t)$ is a solution for $(\ref{gfBBM})$, so that $u(x+y,t)$ is also a solution of the same equation for all $y\in\mathbb{R}$. An important question arises with the presence of the translation symmetry: if the initial value $u_0$ of the associated Cauchy problem for the equation $(\ref{gfBBM})$ is close to the solitary wave $Q_c$, can we conclude that the corresponding evolution $u(x,t)$ is close to the orbit generated by translations associated to $Q_c$? Roughly speaking, we have defined the notion of orbital stability for the solitary wave $Q_c$. More precisely, the formal definition for this concept is given as follows:
\begin{definition}\label{orbital}
		Let $Q_c$ be a traveling wave solution for \eqref{gfBBM}. We say that $Q_c$ is orbitally stable in $H^{\frac{\alpha}{2}}(\mathbb{R})$ provided that, given $\varepsilon>0$, there exists a $\delta>0$ with the following property: if $u_0\in H^{\frac{\alpha}{2}}(\mathbb{R})$ satisfies $\|u_0-Q_c\|_{H^{\frac{\alpha}{2}}}<\delta$, then the local solution, $u(t)$, defined in some interval $[0,T')$ of \eqref{gfBBM} with initial condition $u_0$ can be continued for a solution in $0\leq t<+\infty$ and satisfies
		$$
		\sup_{t\in (0,+\infty)}\inf_{x_0\in\mathbb{R}}||u(\cdot,t)-Q_c(\cdot-x_0)||_{H^{\frac{\alpha}{2}}}<\varepsilon.
		$$
		Otherwise, we say that $Q_c$ is orbitally unstable in $H^{\frac{\alpha}{2}}(\mathbb{R})$.

\end{definition}

\indent In the current study, we discuss the orbital stability of solitary wave solutions of the gfBBM equation in the case of quadratic and cubic non-linearities, that is, $p=1$ and $p=2$, respectively. To illustrate the reason only in the case $p=1$, we can mention that our notion of orbital stability contemplates only local solutions and we can show such property using the approach in Grillakis, Shatah and Strauss \cite{grillakis} for all $c>1$ when $\frac{1}{2}\leq\alpha \leq2$ and for $c>c_1>1$ with  $c_1=\frac{9\alpha+2+\sqrt{2}\sqrt{3\alpha-1}}{15\alpha}$ when $\frac{1}{3}<\alpha < \frac{1}{2}$. We also have orbital stability result for $c_2<c<3/5$ with
 $c_2=\frac{9\alpha+2-\sqrt{2}\sqrt{3\alpha-1}}{15\alpha}$
when $1\leq\alpha \leq2$.

In order to fill  out the lack of stability results we present a spectral (in)stability for the solitary wave $Q_c$. Equation \eqref{gfBBM} can be rewritten in the Hamiltonian form as
\begin{equation}\label{H-form}
	u_t=J H'(u),
\end{equation}
where $J=(I + \frac{5}{4} D^{\alpha})^{-1}\partial_x$ is a skew-symmetric operator and $H$ denotes the Hamiltonian of the corresponding equation given by \eqref{hamilton}.
It is well known that the abstract theory of orbital instability  in \cite{grillakis} can not be applied since $J$ present in the equation $(\ref{H-form})$ is not onto.  We also show that the solitary wave solution is spectrally stable in the same region of orbital stability.

The paper is organized as follows:  In Section 2 we give the global well-posedness result for the Cauchy problem associated with the gfBBM equation.  Section 3 is devoted to the existence of negative solitary waves. We investigate the orbital and spectral stabilities of solitary wave solutions in Section 4. In Section 5, we present some numerical experiments to investigate the stability of the solutions and to understand the behaviour of negative travelling wave solutions.

Throughout  this study, $C$ denotes the generic constant.
\setcounter{equation}{0}

\section{ Global Solutions of the Cauchy Problem for the gfBBM Equation}
The local existence of solutions for the Cauchy problem
\begin{eqnarray}
   && u_t+  u_x + \frac{1}{2}(u^{p+1})_x+   \frac{3}{4}D^{\alpha} u_{x}+ \frac{5}{4}D^{\alpha} u_{t}=0, \label{gfBBM-c}  \\
   && u(x,0)=u_0 (x) \label{incon}
\end{eqnarray}
in the case  \mbox{$0<\alpha<1$} is proved by introducing  the following regularization
 \begin{eqnarray}
 && u_t^{\epsilon} + u_x^{\epsilon} + \frac{1}{2} [(u^{\epsilon})^{p+1}]_x + \frac{3}{4} D ^{\alpha} u_x^{\epsilon}+ \frac{5}{4} D ^{\alpha}u_t^{\epsilon} - \epsilon u^{\epsilon}_{xx} =0,  \label{reg1}  \\
 && u^\epsilon (x,0)=u_0^\epsilon(x).  \label{regin}
\end{eqnarray}

\noindent
The  well-posedness result on the Sobolev space  $H^r(\mathbb{R})$ is given as follows:

\begin{theorem}{(Theorem 2.7 of \cite{oruc})}
Let $0<\alpha <1, r \geq  2 - \frac{\alpha}{2} $ and $u_0 \in H^r(\mathbb{R})$. Then there exists a time $T'>0$, the solution $u^\epsilon$ of  the Cauchy problem for the eqs. \eqref{reg1}-\eqref{regin} converges uniformly to
the unique solution $u$ of the Cauchy problem for the eqs. \eqref{gfBBM-c}-\eqref{incon} in $C([0,T'),H^{r}(\mathbb{R}))$  as $\epsilon\rightarrow 0$.
\end{theorem}

\indent Now, we consider the case   $\alpha > 1$. The Cauchy problem \eqref{gfBBM-c}-\eqref{incon} is rewritten as
\begin{eqnarray}
   && u_t +  ({I+ \displaystyle\frac{3}{4}D^{\alpha}})({I+ \displaystyle\frac{5}{4}D^{\alpha}})^{-1}  u_x  =  - \frac{1}{2}(I+ \displaystyle\frac{5}{4}D^{\alpha})^{-1} (u^{p+1})_x, \label{gfbbm3} \\
   &&  u(x,0)=u_0 (x).   \label{gfbbm4}
\end{eqnarray}
Here, we note that the operator $\displaystyle I+ \frac{5}{4}D^{\alpha}$ is invertible, as its Fourier transform is never zero. The Duhamel formula implies that $u$  is the solution of the Cauchy problem \eqref{gfbbm3}-\eqref{gfbbm4} if, and only if, $u$ satisfies the integral equation $u=\Phi u $  where

\begin{equation}\label{reg2}
  (\Phi u)(x,t) = S(t) u_0(x) - \frac{1}{2} \int_{0}^{t} S({t-\tau}) \big[ {\partial_x(I+ \displaystyle\frac{5}{4}D^{\alpha})^{-1}}u^{p+1} \big] (x,\tau) d\tau,
\end{equation}
with
\begin{equation*}
  S(t) u= \mathcal{F}^{-1}\bigg( e^{-\frac{1+ \frac{3}{4}|\xi|^{\alpha}}{1+ \frac{5}{4}|\xi|^{\alpha}} i \xi t} \bigg)* u(x).
\end{equation*}
The symbol $*$ denotes the convolution operation.
\begin{lemma}{(\cite{adams})} \label{lemproduct}
For $r > 1/2$,  $H^r (\mathbb{R})$ is an algebra with respect to the
product of functions. That is, if $u, v \in H^r(\mathbb{R})$ then   $uv \in H^r(\mathbb{R})$  and
\begin{equation}
\|uv \|_ {H^r (\mathbb{R})}  \leq C \|u\|_ {H^r (\mathbb{R})} \| v \|_ {H^r (\mathbb{R})}.
\end{equation}
\end{lemma}


Now, we  prove the existence and uniqueness of the local solution for the problem \eqref{gfBBM-c}-\eqref{incon} with $\alpha > 0$ by using the contraction mapping principle.

\begin{theorem}\label{lwp1}
Let $\alpha > 1$ be fixed. Consider an initial data  $u_0 \in H^r (\mathbb{R})$ with $r\geq\alpha/2$. There exists a time $T'=T'(||u_0||_{H^r})>0$ and a unique local solution $u \in C([0,T'),H^r(\mathbb{R}))$ to the Cauchy problem
\eqref{gfBBM-c}-\eqref{incon}.
\end{theorem}
\begin{proof}
For $T>0$ and $a>0$, define $\overline{B}_{T,a}$ as the closed ball
\begin{equation*}
  \overline{B}_{T,a}=\{u \in C([0,T], H^r(\mathbb{R})), \|u\|_{L^{\infty}([0,T],H^r(\mathbb{R}))} \leq 2a\}.
\end{equation*}
We first prove  $\Phi$ maps $\overline{B}_{T,a}$ into $\overline{B}_{T,a}$ for $T>0$ small enough and a convenient choice of $a>0$. From \eqref{reg2}, one gets
\begin{equation}\label{reg3}
  \|\Phi u(t)\|_{H^r}
  \leq   \|S(t)u_0\|_{H^r} +\frac{1}{2}  \int_{0}^{t} \| S({t-\tau}) \big[ {\partial_x}{(I+\frac{5}{4}D^{\alpha})^{-1}} u^{p+1}(\tau) \big] \|_{H^r}  d\tau.
\end{equation}
The definition of Sobolev norm allows us to estimate the following term
\begin{equation}\label{Gtu0}\begin{array}{lllll}
 \|S(t)u_0\|_{H^r} &=&
 \|(1+|\xi|^2)^{r/2}  e^{-\frac{1+ \frac{3}{4}|\xi|^{\alpha}}{1+ \frac{5}{4}|\xi|^{\alpha}} i \xi t} \hat{u}_0(\xi)\| \\
&\leq&   \displaystyle  \|u_0\|_{H^r}.
\end{array}
\end{equation}
\noindent
Using Lemma \ref{lemproduct} and $1+\displaystyle \frac{5}{4} |\xi|^{\alpha}  \geq |\xi|$  for $\alpha \geq 1$, one gets similarly to $(\ref{Gtu0})$ that
\begin{eqnarray}  \label{reg4}
  \| S({t-\tau}) \big[{\partial_x}{(I+\frac{5}{4}D^{\alpha})^{-1}} u^{p+1}(t)\big] \|_{H^r}
&\leq& \displaystyle \| (1+|\xi|^2)^{r/2} \frac{i\xi}{1+\frac{5}{4} |\xi|^{\alpha}} \widehat{u^{p+1}}(\xi,t) \|   \nonumber \\
&=&\displaystyle  \| u^{p+1}(t)\|_{H^r} \nonumber \\
&\leq&\displaystyle   C \| u(t)\|_{H^r}^{{p+1}}.
\end{eqnarray}
Using \eqref{Gtu0} and \eqref{reg4} in \eqref{reg3}, we obtain
\begin{equation*}
  \|\Phi u (t)\|_{H^r} \leq \|u_0\|_{H^r} + \frac{C T}{2} \underset{t \in [0,T] }{\mbox{sup}} \|u(t)\|_{H^r}^{p+1}=\|u_0\|_{H^r}+C T2^pa^{p+1}.
\end{equation*}
By choosing $a=||u_0||_{H^r}$ and $T=T'>0$ small enough and satisfying  $T \leq \displaystyle\frac{1}{C2^{p}\|u_0\|_{H^r}^p}$, we have that $\Phi$ maps $\overline{B}_{T,a}$ into $\overline{B}_{T,a}$.\\
\indent We show that  $\Phi$ is a strict contraction. Let $u_1$, $u_2$ $\in$ $\overline{B}_{T,a}$. The Duhamel formula  \eqref{reg2} gives
\begin{equation}
  \|\Phi u_1(t)- \Phi u_2 (t)\|_{H^r}    \leq  \frac{1}{2} \int_{0}^{t} \| S({t-\tau}) {\partial_x}{(I+\frac{5}{4}D^{\alpha})^{-1}} \big[u_1^{p+1}-u_2^{p+1} \big](\tau) \|_{H^r} d\tau .
\end{equation}
An analogous argument as in \eqref{reg4} implies
\begin{eqnarray*}
\|\Phi u_1(t)- \Phi u_2(t)\|_{H^r}  &\leq &  \frac{C}{2} \int_{0}^{t}    \| u_1^{p+1}(\tau)- {u_2}^{p+1}(\tau) \|_{H^r} d\tau   \nonumber\\\\
  &\leq & C T 2^{p-1}a^p \underset{t \in [0,T] }{\mbox{sup}} \|u_1(t)-u_2(t)\|_{H^r}.  \nonumber
\end{eqnarray*}
Choosing a smaller  $T>0$ as above such that $\displaystyle T \leq \frac{1}{C2^p||u_0||_{H^r}^p}<\frac{2}{C2^p||u_0||_{H^r}^p}$, we have that  $\Phi$ is strictly contractive. The remainder of the proof relies on application of the contraction mapping principle but we omit the details.
\end{proof}


\indent Theorem $\ref{lwp1}$ and the conservation of the quantities $H$ and $F$ give us the following result.

\begin{theorem}\label{teogwp}
	Let $\alpha>1$ be fixed. For each $u_0\in H^{\frac{\alpha}{2}}(\mathbb{R})$, the maximal local time of existence $T'$ in Theorem $\ref{lwp1}$ can be considered as $T'=+\infty$, that is, $u\in C([0,+\infty);H^{\frac{\alpha}{2}}(\mathbb{R}))$.	
\end{theorem}
\begin{flushright}
	$\square$
\end{flushright}

\setcounter{equation}{0}
\section{Existence of Negative Solitary Wave Solutions}

\indent A solitary wave solution associated to $(\ref{gfBBM})$ is a smooth solution of the form $u(x,t)=Q_c(\xi)$, $\xi=x-ct$ with wave speed $c>0$ and satisfying the decay property $\displaystyle{\lim_{|\xi| \rightarrow \infty} Q_c^{(n)}(\xi)=0}$, $n\in\mathbb{N}$. Using this ansatz on \eqref{gfBBM}
we obtain,
\begin{equation} \label{sweq1}
	-cQ_c'+Q_c'+\frac{1}{2}(Q_c^{p+1})'- \frac{5}{4} c D^{\alpha} Q_c' + \frac{3}{4} D^{\alpha} Q_c'=0.
\end{equation}
Integrating on $\mathbb{R}$, the above equation reduces to the following ODE
\begin{equation}\label{sw1}
	(\frac{5}{4} c-\frac{3}{4})D^{\alpha}Q_c   +(c-1) Q_c -\frac{1}{2} Q_c^{p+1} =0.
\end{equation}

In \cite{oruc}, it is stated that the equation \eqref{gfBBM} has nontrivial solutions when $c<3/5$ or $c>1$ and $\alpha > p/(p+2)$.  The  existence and uniqueness of  positive solitary waves  $Q_c \in H^{\frac{\alpha}{2}}(\mathbb{R})$ for $c>1$ and $\alpha > \frac{p}{p+2}$ are also proved in \cite{oruc}. We note that even though only the case $0<\alpha<1$ is considered in \cite{oruc},   the theory  in \cite{franklenzmann} gives the existence of positive solitary waves for $0<\alpha<2$.

Now we discuss the existence of negative solitary wave solutions.
According to the Pohozaev type identity
\begin{equation}\label{energy1}
  \big(\frac{5}{4} c-\frac{3}{4} \big) \int_{\mathbb{{R}}} |D^{\frac{\alpha}{2}} Q_c|^2 dx + (c-1)\int_{\mathbb{{R}}} Q_c^2 dx = \frac{1}{2} \int_{\mathbb{{R}}} Q_c^{p+2} dx
\end{equation}
given in \cite{oruc}, it is possible to say that the identity is satisfied for a negative $Q_c$ when $c<3/5$ and $p$ is odd. To show the existence of such waves we use the result of \cite{franklenzmann}. Setting $Q_c=-R_c$  where $R_c>0$  the equation \eqref{sw1} becomes
\begin{equation}\label{sw2}
 (\frac{3}{4} -\frac{5}{4}c)D^{\alpha}R_c   +(1-c) R_c -\frac{1}{2} R_c^{p+1} =0.
\end{equation}
A scaling argument as
$
R_c(\xi)=(2(1-c))^{1/p}Q\left(\left( \frac{4(c-1)}{5c-3}  \right)^{1/\alpha}   \xi \right)
$
converts \eqref{sw2} into the equation
\begin{equation}\label{ground}
D^{\alpha} Q+Q-Q^{p+1}=0.
\end{equation}
\indent The existence and uniqueness of an even and positive solution $Q\in H^{\alpha/2}(\mathbb{R})$ for \eqref{ground} has been determined in \cite{franklenzmann} for $0<\alpha<2$ and  $0<p<p_{max}$. Here, $p_{\max}$ is the critical exponent given by
\begin{equation}\label{pcond}
p_{\max}(\alpha)=\left\{ \begin{array}{cc}
\frac{2\alpha}{1-\alpha}, &\mbox{for}~~ 0< \alpha < 1  \\
\infty, &\mbox{for}~~ 1\leq \alpha < 2.   \end{array} \right.
\end{equation}
Therefore, we obtain the existence and uniqueness of a positive even solution $R_c$ of the equation \eqref{sw2} and consequently a negative solution $Q_c$ of the equation \eqref{sw1} when $c<3/5$ and $p$ is odd for $0< \alpha < 2$. In Figure 1, we depict existence of solitary waves for different values of $\alpha$ and $c$ with $p=1$. \\
\indent For the case of $\alpha=2$ and $p=1$ the equation \eqref{gfBBM} has the exact solution
\begin{equation}\label{exact}
  u(x,t) = 3({c-1}) \mbox{sech}^2\left(\frac{1}{2} \sqrt{\frac{4(c-1)}{ 5c -3}}(x-ct)\right).
\end{equation}
It is clear that $(\ref{exact})$ agrees with the following statement: the solution is positive when $c>1$ and it is negative when $c<3/5$.

\indent To finish, we note that if $Q_c(x-ct)>0$ is a solitary wave solution for the gfBBM equation, then $-Q_c(x-ct)$ is also a solution when $c>1$ and $p$ is even. However, we do not consider these negative solutions in the current study since  the dynamics for the solutions $Q_c$ and $-Q_c$ are the same.

\begin{figure}[h!]
\centering
\begin{tikzpicture}[scale=5][baseline=0pt]
 \draw[->] (0,0) -- (2.3,0) node[right]{$\alpha$};
 \draw[->] (0,0) -- (0,1.3) node[left]{$c$};
 \draw  (-0.1,0)  node[below]{$0$};
 \draw  (0.33,0)  node[below]{$1/3$};
 \draw  (2,0)  node[below]{$2$};
 \draw  (0,0.6)  node[left]{$3/5$};
 \draw  (0,1)  node[left]{$1$};

\path [fill=lightgray!40] (0,0.6) rectangle (2,1);
\path [fill=red!20!] (0,0) rectangle (0.33,1.22);

\draw[dashed] (2,0)--(2,1.3);

\draw[dashed] (0,1)--(2,1);
\draw[dashed] (0,0.6)--(1.3,0.6);

\path [fill=yellow!50] (0.33,0) rectangle (2,0.6);
\draw[dashed] (0.33,0)--(0.33,0.6);
\draw[dashed] (0.33,0.6)--(2,0.6);

\draw[dashed] (0.33,0)--(0.33,1.3);
\node [above] at (1,0.8) {no nontrivial };
\node [below] at (1.03,0.8) {solitary wave  solution};

\node [below] at (1,0.35) {negative solitary wave};
\node [above] at (1,1.05) {positive solitary wave };

\node [above] at (0.15,0.05) {\rotatebox{90} {Hamiltonian is not well-defined}};
\end{tikzpicture}
\caption{Dependence of the existence of solutions on $\alpha$ and $c$ when $p=1$.} \label{table1}
\end{figure}
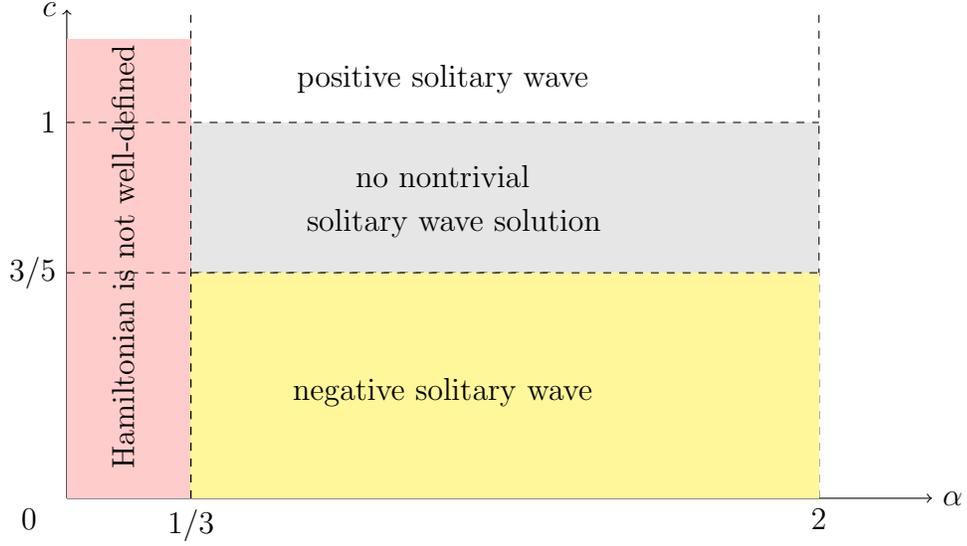

\setcounter{equation}{0}
\section{Spectral and Orbital Stability of Solitary Waves}
In this section, we study the spectral and orbital stability of solitary waves.
To obtain a precise definition of spectral stability, we need to add a perturbation $v$ to the smooth travelling wave $Q_c$ propagating with the same fixed speed $c$ of the form \mbox{$u(x,t) = Q_c(x-ct) + v(x-ct,t)$} in $(\ref{gfBBM})$. Using that $Q_c$ satisfies $(\ref{sw1})$ and performing a truncating by the linear terms in $v$, we obtain  evolution problem
\begin{equation}
	\label{gf-lin}
	v_t =  J \mathcal{L}_c v,
\end{equation}
where $\mathcal{L}_c$ is the second order differential linearized operator
in the form
\begin{equation}
	\mathcal{L}_c
	=\left(\frac{5}{4} c - \frac{3}{4}\right)D^{\alpha}+c-1 -\frac{p+1}{2}Q_c^p.
	\label{hill}
\end{equation}
\indent For $\lambda\in\mathbb{C}$, we consider a growing mode solution of the form $v(x,t)=e^{\lambda t}w(x)$ which yields
\begin{equation}
	\label{spect-prob}
	J \mathcal{L}_c w=\lambda w.	
\end{equation}
Denoting the spectrum of $J\mathcal{L}_c$ by $\sigma(J\mathcal{L}_c)$, the spectral stability of the solitary wave $Q_c$ is defined as follows:

\begin{definition}\label{defistab}
	The solitary wave $Q_c \in H^{\infty}(\mathbb{R})$ is said to be spectrally stable if $\sigma(J\mathcal{L}_c) \subset i\mathbb{R}$ in  $L^2(\mathbb{R})$. Otherwise, that is, if $\sigma(J\mathcal{L}_c)$ in $L^2(\mathbb{R})$ contains a point $\lambda$ with $Re(\lambda)>0$, the solitary wave $Q_c$ is said to be 	spectrally unstable.
\end{definition}

\begin{obs} Following the arguments quoted in the paragraph below \cite[Definition 1.3]{pava}, we see that $(\ref{gf-lin})$ is linearized equation without forcing or damping terms. This kind of evolution model gives, from $(\ref{spect-prob})$, certain symmetries on the spectrum of $\sigma(\mathcal{L}_c)$. Indeed, $\sigma(\mathcal{L}_c)$ will be symmetric with respect to the reflection in the real and imaginary axes so that, it implies that exponentially growing perturbations are always paired with exponentially decaying ones. It is the reason why, in the Definition $\ref{defistab}$, only the spectral parameter $\lambda$ satisfying $Re(\lambda) > 0$ was required.
	
	\end{obs}

The spectral stability  can be studied for the case of smooth solutions which are global in time.
Since $w$ solves $(\ref{spect-prob})$, a bootstrap argument can be performed to conclude the smoothness of $w$, so that the evolution of $u$ is also smooth as required. Since the solution $u$ can be considered smooth, we can conclude that the solutions handled in the spectral stability can be considered global in time by Theorem $\ref{teogwp}$.\\
\indent Next, associated to the conservation of mass (\ref{cq2}), we have the basic orthogonality
condition on the perturbation $v \in H^{\alpha}(\mathbb{R})$
to the smooth solitary wave $Q_c \in H^{\infty}(\mathbb{R})$
given by $
	\langle F'(Q_c), v \rangle = 0
$
and it is preserved in the time evolution of the linearized equation
(\ref{gf-lin}). Moreover, it is well known that the spectral stability of the solitary wave
$Q_c$ holds if the linearized operator $\mathcal{L}_c$ is positive
on the orthogonal complement of the spanned subspace $[F'(Q_c)]$
in $L^2(\mathbb{R})$. As it is well-known and assuming that $\ker(\mathcal{L}_c)=[Q_c']$,
the positivity property holds if the number of negative eigenvalues
of the quantity
\begin{equation}\label{matrix}
		\mathcal{I} :=
		\left\langle Q_c + \frac{5}{4}D^{\alpha}Q_c, \mathcal{L}_c^{-1} \left(Q_c + \frac{5}{4}D^{\alpha}Q_c\right)\right\rangle
\end{equation}
coincides with the number of negative eigenvalues of $\mathcal{L}_c$ in $L^2(\mathbb{R})$. Now, since $c\mapsto Q_c$ is a smooth curve of solitary waves, it follows by differentiating (\ref{sw1}) in $c$ that \begin{equation}
	\label{relationphicb}\mathcal{L}_c\left(\partial_cQ_c\right)=-\left(Q_c + \frac{5}{4}D^{\alpha}Q_c\right).\end{equation}
Hence, the quantity $\mathcal{I}$ in (\ref{matrix}) can be rewritten as
\begin{equation}
	\label{matrix-P}
	\mathcal{I}= -\frac{d}{dc}F(Q_c).
\end{equation}
\indent Let us assume that $\ker(\mathcal{L}_c)=[Q_c']$. According with \cite[Theorem 5.3.1]{KP}, we can calculate the exact number of negative eigenvalues of $\mathcal{L}_c|_{\{F'(Q_c)\}^{\bot}}$ by using the index formula as

\begin{equation}
	\label{index-f}
	n\left(\mathcal{L}_c|_{\{F'(Q_c)\}^{\bot}}\right)=n(\mathcal{L}_c)-n(\mathcal{I}),
\end{equation}
where $n(\mathcal{A})$ stands the number of negative eigenvalues of a certain linear operator $\mathcal{A}$ counting multiplicities. Therefore, if $n(\mathcal{L}_c)=n(\mathcal{I})=1$ one has the required spectral stability.\\
\indent To obtain the orbital (nonlinear) stability in the sense of Definition $\ref{defistab}$, we need the following basic result:
\begin{lemma}
	\label{teoest}
	Let $Q_c \in H^{\infty}(\mathbb{R})$ be a solitary wave
	of the equation (\ref{sw1}). Assume that $n(\mathcal{L}_c)=1$ and $\ker(\mathcal{L}_c)=[Q_c']$.
	Suppose that $\Phi=\partial_c Q_c\in H^{\alpha}(\mathbb{R})$ satisfies
	$\langle\mathcal{L}_c\Phi,\Phi\rangle<0$
	and $\langle\mathcal{L}_c\Phi,v\rangle=0$, for all $v \in \Upsilon_0$,
	where
	\begin{equation*}
		\Upsilon_0=\{  v \in H^{\frac{\alpha}{2}}(\mathbb{R}) :\ \ \langle F'(Q_c), v \rangle=0\},
	\end{equation*}
	where $F$ is the conserved quantity associated with the equation (\ref{cq2}).
	Thus, the solitary wave $Q_c$ is orbitally stable in $H^{\frac{\alpha}{2}}(\mathbb{R})$. The same result is valid when considering the negative solitary wave $R_c$ instead of $Q_c$ and the linearized operator
\begin{equation}
	\mathcal{L}_c^-
	=\left(\frac{3}{4}-\frac{5}{4} c \right)D^{\alpha}+1-c -\frac{p+1}{2}Q_c^p
	\label{hill-2}
\end{equation}
instead of $\mathcal{L}_c$.
\end{lemma}
\begin{proof}
	The proof of this result is an adaptation of the general orbital stability approach in \cite{grillakis} and \cite[Section 2]{weinstein}. For additional references, we can cite  \cite[Chapter 6]{angulo} and \cite[Section 4]{natalipastor} regarding the Korteweg-de Vries/Schr\"odinger type equations. All mentioned results can be adapted to our model without further problems. Indeed, first we have $n(\mathcal{L}_c)=1$ and $\ker(\mathcal{L}_c)=[Q_c']$. For $\Phi=\partial_cQ_c$, one has from $(\ref{relationphicb})$ that $\mathcal{L}_c\Phi=-\left(Q_c+\frac{5}{4}D^{\alpha}Q_c\right)$ and since $\langle\mathcal{L}_c\Phi,\Phi\rangle=\mathcal{I}=-\frac{d}{dc}F(Q_c)$, we obtain the orbital stability in the sense of Definition $\ref{defistab}$ provided that $\mathcal{I}<0$ (or equivalently, $d''(c):=\frac{d}{dc}F(Q_c)=-\mathcal{I}>0$ as required in \cite{grillakis}). Gathering all informations, we have
	\begin{equation}\label{positivity}
		\langle\mathcal{L}_cv,v\rangle \geq C||v||_{H^{\frac{\alpha}{2}}}^2,
		\end{equation}
	for all $v\in \Upsilon_0\cap [Q_c']^{\bot}$. The positivity in $(\ref{positivity})$ allows us to conclude the orbital stability in the energy space $H^{\frac{\alpha}{2}}(\mathbb{R})$ since $Q_c$ is a minimum of the energy functional $H$ in $(\ref{hamilton})$ for fixed momentum $F$ in $(\ref{cq2})$.
\end{proof}

 \begin{obs} Using Lemma $\ref{teoest}$ we can affirm, since  $n(\mathcal{L}_c)=n(\mathcal{I})=1$, that \textit{the spectral stability implies the orbital stability in our case and vice and versa}.\end{obs}

 \indent Next, for any $k\in\mathbb{N}$, if $n\left(\mathcal{L}_c|_{\{F'(Q_c)\}^{\bot}}\right)=2k-1$ we have at least one eigenvalue with positive real part for the operator $\partial_x\mathcal{L}_c$ and then, the solitary wave is spectrally unstable (see \cite{pava}, \cite{KS} and references therein).\\
\indent In order to calculate $n\left(\mathcal{L}_c|_{\{F'(Q_c)\}^{\bot}}\right)$ using identity $(\ref{index-f})$, we first need to know the behavior of the non-positive eigenvalues of $\mathcal{L}_c$. Since the equation $(\ref{gfBBM})$ is invariant under the symmetry of translation, we have at least that $Q_c'\in\ker(\mathcal{L}_c)$. However this information is not enough to calculate $\mathcal{I}$ in $(\ref{matrix})$ since we do not know if $\mathcal{L}_c^{-1}\left(Q_c+\frac{5}{4}D^{\alpha}Q_c\right)$ is well defined. Next result give us in fact that $n(\mathcal{L}_c)=z(\mathcal{L}_c)=1$, where  $z(\mathcal{L}_c)$ indicates the dimension of $\ker(\mathcal{L}_c)$.\\

\begin{proposition}
	\label{theolinop}
	Let  $\alpha\in\left(0,2\right]$ be fixed and consider $0<p<p_{max}$, with $p_{max}$ defined as in $(\ref{pcond})$. For $c>1$, the linearized self-adjoint operator $\mathcal{L}_c$ in (\ref{hill}) defined in $L^2(\mathbb{R})$, with the dense domain $H^{\alpha}(\mathbb{R})$ satisfies $n(\mathcal{L}_c)=z(\mathcal{L}_c)=1$. Moreover, the essential spectrum of $\mathcal{L}_c$  is given by $\sigma_{{\rm
			ess}}(\mathcal{L}_{c})=\left[c-1,+\infty\right)$.
\end{proposition}
  \begin{proof}
  	According with \cite{franklenzmann}, let $\phi$ be the unique solitary wave profile associated with the equation
  	
  	\begin{equation}\label{fl}
  		D^{\alpha}{\phi}+ {\phi}-{\phi}^{p+1}=0.
  	\end{equation}
 For $ \theta = \big(\frac{4(c-1)}{5c-3}\big)^{1/{\alpha}}$, we have that  $Q_c$ given by \begin{equation}
 	\label{transf}
 Q_c(x)=\big({2(c-1)}\big)^{1/p}{\phi}\big( \theta x \big)\end{equation} is also the unique solitary wave which solves equation $(\ref{sw1})$.\\
 \indent Next, consider the linearized operator associated with the solitary wave $\phi$
  \begin{equation}\label{operafrac}
 \mathcal{P} = D^{\alpha}+1-(p+1){\phi^p}.
\end{equation}
According with the arguments in \cite{franklenzmann}, we have that $\mathcal{P}$ is a self-adjoint operator defined in $L^2(\mathbb{R})$ with dense domain $H^{\alpha}(\mathbb{R})$, the essential spectrum is $\sigma_{ess}(\mathcal{P})=[1,+\infty)$ and $n(\mathcal{P})=z(\mathcal{P})=1$. We follow the idea given by \cite{pava} in order to relate both operators $\mathcal{P} $ and $\mathcal{L}_c $. In fact, using $(\ref{transf})$
 and the dilation operator
$(T_{\theta}f)(x)=f(\theta x),$ we have that
\begin{equation}\label{operators1}
  \mathcal{L}_c= (c-1) T_{\theta} \mathcal{P} T_{\theta}^{-1}.
\end{equation}
Relation $(\ref{operators1})$ implies that the operators $\mathcal{P} $ and $\mathcal{L}_c $ have the same spectral structure,
i.e. $spec(\mathcal{L}_c)=\left\{(c-1)\lambda: \lambda\in spec(\mathcal{P})\right\}$. Since both operators are self-adjoint we have that  $n(\mathcal{L}_c)=z(\mathcal{L}_c)=1$ and $\sigma_{{\rm
		ess}}(\mathcal{L}_{c})=\left[c-1,+\infty\right)$ as requested.
\end{proof}

The spectrum of the linearized operator for the negative solitary wave solutions is given with the following Proposition:

\begin{proposition}
	\label{theolinop-2}
	Let  $\alpha\in\left(0,2\right]$ be fixed and consider $0<p<p_{max}$, with $p$ odd and $p_{max}$ defined as in $(\ref{pcond})$. For $c<3/5$, the linearized self-adjoint operator $\mathcal{L}_c^-$
defined in $L^2(\mathbb{R})$, with the dense domain $H^{\alpha}(\mathbb{R})$ satisfies $n(\mathcal{L}_c^-)=z(\mathcal{L}_c^-)=1$. Moreover, the essential spectrum of $\mathcal{L}_c^-$  is given by $\sigma_{{\rm			ess}}(\mathcal{L}_{c}^-)=\left[1-c,+\infty\right)$.
\end{proposition}

\begin{proof}
  We can relate the operator $\mathcal{L}_c^-$ with $\mathcal{P}$ with the dilation operator $(T_{\theta}f)(x)=f(\theta x)$ where  $\theta = \big(\frac{4(c-1)}{5c-3}\big)^{1/{\alpha}}$.  We have that
\begin{equation}\label{operators}
  \mathcal{L}_c^-= (1-c) T_{\theta} \mathcal{P} T_{\theta}^{-1}.
\end{equation}
Relation $(\ref{operators})$ implies that the operators $\mathcal{P} $ and $\mathcal{L}_c ^-$ have the same spectral structure,
i.e.   $spec(\mathcal{L}_c^-)=\left\{(1-c)\lambda: \lambda\in spec(\mathcal{P})\right\}$.
The rest of the proof is the same as in Proposition \ref{theolinop}.
\end{proof}

\indent We have a convenient result for the quantity $\mathcal{I}$ given by $(\ref{matrix-P})$
\begin{lemma}
	\label{theoF}
Let  $p>0$ be fixed and consider $\alpha>\frac{p}{p+2}$. For $c\in \left(0,\frac{3}{5}\right)\cup (1,+\infty)$, we have
	\begin{equation*}
		\mathcal{I}=-\frac{d}{dc} \left[(c-1)^{2/p} \bigg(\frac{5c-3}{4(c-1)}\bigg)^{1/{\alpha}}  \left( 1 + \frac{5p(c-1)}{(5c-3)(\alpha(p+2) -p)} \right) \right]\| {\phi} \|^2.
	\end{equation*}
	\end{lemma}

\begin{proof}
Similar arguments as performed in \cite[Theorem 4.1]{oruc} give us that
\begin{equation}\label{pohoz}
  \int_{-\infty}^{+\infty} |D^{\frac{\alpha}{2}} Q_c(\xi)  |^2 d\xi = \frac{4p(c-1)}{(5c-3)(\alpha(p+2)-p)}  \int_{-\infty}^{+\infty} Q_c^2(\xi) d\xi.
\end{equation}
Transformation $(\ref{transf})$ applied in the last integral of $(\ref{pohoz})$ enables us to conclude
\begin{equation}\label{scalL2}
  \int_{-\infty}^{+\infty}  |Q_c(\xi)|^2 d\xi  =  \frac{\big( 2(c-1)\big)^{2/p} (5c-3) }{4(c-1)^{1/\alpha}} \int_{-\infty}^{+\infty} |{\phi}({\xi})|^2 d{\xi}.
\end{equation}
Finally, combining $(\ref{pohoz})$ and $(\ref{scalL2})$ we obtain
\begin{eqnarray*}
\nonumber
\mathcal{I}&=&- \frac{d}{dc}F(Q_c)=-\frac{1}{2}\frac{d}{dc}\int_{-\infty}^{+\infty}\left(Q_c^2+\frac{5}{4}|D^{\frac{\alpha}{2}}Q_c|^2\right)d\xi\\
\nonumber
   &=&- 2^{2/p-1}\frac{d}{dc} \left[(c-1)^{2/p} \bigg(\frac{5c-3}{4(c-1)}\bigg)^{1/{\alpha}}  \left( 1 + \frac{5p(c-1)}{(5c-3)(\alpha(p+2) -p)} \right) \right]\| {\phi} \|^2. \\
   &  ~~&
\end{eqnarray*}
\end{proof}

\indent Let  $p>0$ be fixed and consider $\alpha>\frac{p}{p+2}$. Define $\mathcal{K}:(0,\frac{3}{5})\cup(1,+\infty)\rightarrow \mathbb{R}$ given by
\begin{equation}\label{derKc1}\mathcal{K}(c)= \left[(c-1)^{2/p} \bigg(\frac{5c-3}{4(c-1)}\bigg)^{1/{\alpha}}  \left( 1 + \frac{5p(c-1)}{(5c-3)(\alpha(p+2) -p)} \right) \right].
\end{equation}

\noindent When  $p=1$ the derivative of $\mathcal{K}$ has the two roots in terms of $c$ given by
\begin{equation}\label{zerocp1}
	c_1 = \frac{9\alpha+2+\sqrt{2}\sqrt{3\alpha-1}}{15\alpha},~~~~~~~~
~   c_2 = \frac{9\alpha+2-\sqrt{2}\sqrt{3\alpha-1}}{15\alpha}.
\end{equation}
\indent The variations of the roots $c_1$ and $c_2$  with $\alpha$ are presented in the left panel of Figure 2.

\indent We have the following scenarios:\\\\
\indent (i) When  $c>1$, there exist positive solitary wave solutions. If $\alpha>0$, then $c_2<1$. On the other hand, we see that $c_1>1$ for $\alpha\in\left(\frac{1}{3},\frac{1}{2}\right)$.\\\\
\indent (ii) When  $0<c<3/5$, there exist negative solitary wave solutions. For $0<\alpha<2$ we have $c_1>3/5$. On the other hand, it follows that $c_2<3/5$ for $\alpha\in\left(1,+\infty\right)$.\\\\

\begin{figure}
 \begin{minipage}[h]{0.45\linewidth}
   \includegraphics[scale=0.6]{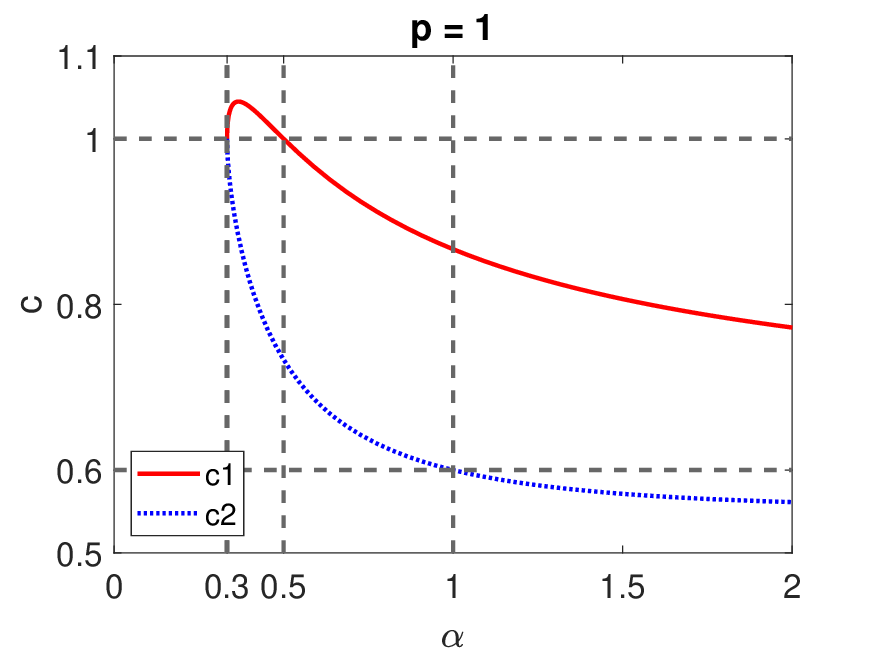}
 \end{minipage}
\hspace{30pt}
\begin{minipage}[h]{0.45\linewidth}
   \includegraphics[scale=0.6]{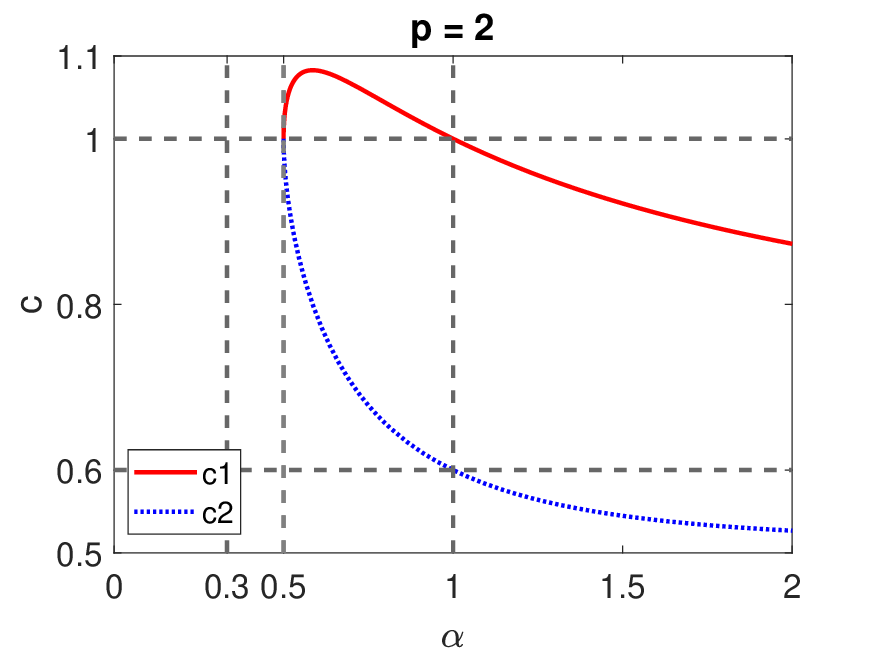}
 \end{minipage}
  \caption{\small Variations of the roots of the derivative of $\mathcal{K}$ with $\alpha$ for $p=1$ and $p=2$.}
\end{figure}
\indent Now we investigate the sign of $\frac{d}{dc}\mathcal{K}(c)$ to determine the stability:
\begin{itemize}
  \item  If $\frac{1}{3}<\alpha\leq \frac{1}{2}$, one has that $\frac{d}{dc}\mathcal{K}(c)<0$ for $c\in (0,\frac{3}{5})\cup(1,c_1)$ and $\frac{d}{dc}\mathcal{K}(c)>0$ for $c>c_1$.
    \item  If $\frac{1}{2}<\alpha\leq 1$, there is no zeros for $\frac{d}{dc}\mathcal{K}(c)$ in the set $c \in (0,\frac{3}{5})\cup(1,+\infty)$.  Here, $\frac{d}{dc}\mathcal{K}(c)<0$ for all $c\in(0,\frac{3}{5})$ and $\frac{d}{dc}\mathcal{K}(c)>0$ for all $c>1$.
  \item  If $\alpha>1$, one has that $\frac{d}{dc}\mathcal{K}(c)>0$ for $c\in (c_2,\frac{3}{5}) \cup (1, +\infty)$ and $\frac{d}{dc}\mathcal{K}(c)<0$ for $c\in (0,c_2)$.
\end{itemize}

For the case $\frac{1}{3}<\alpha\leq \frac{1}{2}$, analysis above, Lemma $\ref{teoest}$ and Proposition $\ref{theolinop}$ enable us to conclude that the solitary wave $Q_c$ is spectrally (orbitally) stable if $c>c_1$. The solitary wave $Q_c$ is spectrally unstable in the case $1<c<c_1$ or $0<c<3/5$. The same arguments can be used for the case $\alpha\in(\frac{1}{2},1]$. In fact, the solitary wave $Q_c$ is spectrally stable for all $c>1$ and spectrally unstable for all $0<c<3/5$.  Finally when  $\alpha\in(1,2]$, the solitary wave $Q_c$ is spectrally stable for all $c>1$ or $c_2<c<3/5$ and spectrally unstable for $0<c<c_2$.\\
We present the results in Figure 3.

\begin{figure}
\centering
\begin{tikzpicture}[scale=6][baseline=0pt]
 \draw[->] (0,0) -- (2.3,0) node[right]{$\alpha$};
 \draw[->] (0,0) -- (0,1.3) node[left]{$c$};
 \draw  (-0.1,0)  node[below]{$0$};
 \draw  (0.33,0)  node[below]{$1/3$};
 \draw  (0.5,0)  node[below]{$1/2$};
 \draw  (1,0)  node[below]{$1$};
 \draw  (2,0)  node[below]{$2$};
 \draw  (0,0.6)  node[left]{$3/5$};
 \draw  (0,1)  node[left]{$1$};

\path [fill=lightgray!40] (0,0.6) rectangle (2,1);
\path [fill=red!20!] (0,0) rectangle (0.33,1.22);


\draw[dashed] (0,1)--(2,1);
\draw[dashed] (2,0)--(2,1.3);
\draw[dashed] (0,0.6)--(2,0.6);

\draw[dashed] (0.33,0)--(0.33,0.6);
\draw[dashed] (0.33,0.6)--(1,0.6);

\draw[dashed] (0.33,0)--(0.33,1.3);
\node [above] at (0.8,0.8) {no nontrivial };
\node [below] at (0.83,0.8) {solitary wave  solution};

\node [below] at (0.85,0.35) {negative solitary wave};
\node [above] at (0.85,1.05) {positive solitary wave };

\draw[dashed] (1,0)--(1,1.3);
\draw[dashed] (1/2,0)--(1/2,1.3);
\node [above] at (0.15,0.25) {\rotatebox{90} {Hamiltonian is not well-defined}};

\draw[thick, domain=0.3334:0.5] plot (\x, {(-4*\x-2+3*\x*sqrt(6*\x-2))/(-10*\x+5*\x*sqrt(6*\x-2))});
\draw[thick, domain=1:2] plot (\x, {(4*\x+2+3*\x*sqrt(6*\x-2))/(10*\x+5*\x*sqrt(6*\x-2))});
\draw [line width=0.7]  (1,0.6)--(2,0.6);
\draw [line width=0.7]  (2,0.56)--(2,0.6);
\draw [line width=0.7]  (2,1)--(2,1.3);
\draw [line width=0.7]  (0.5,1)--(2,1);
\draw [line width=0.7] (0.33,1)--(0.33,1.3);

\end{tikzpicture}
\caption{The cases for the existence of solutions when $p=1$.} \label{table12}
\end{figure}
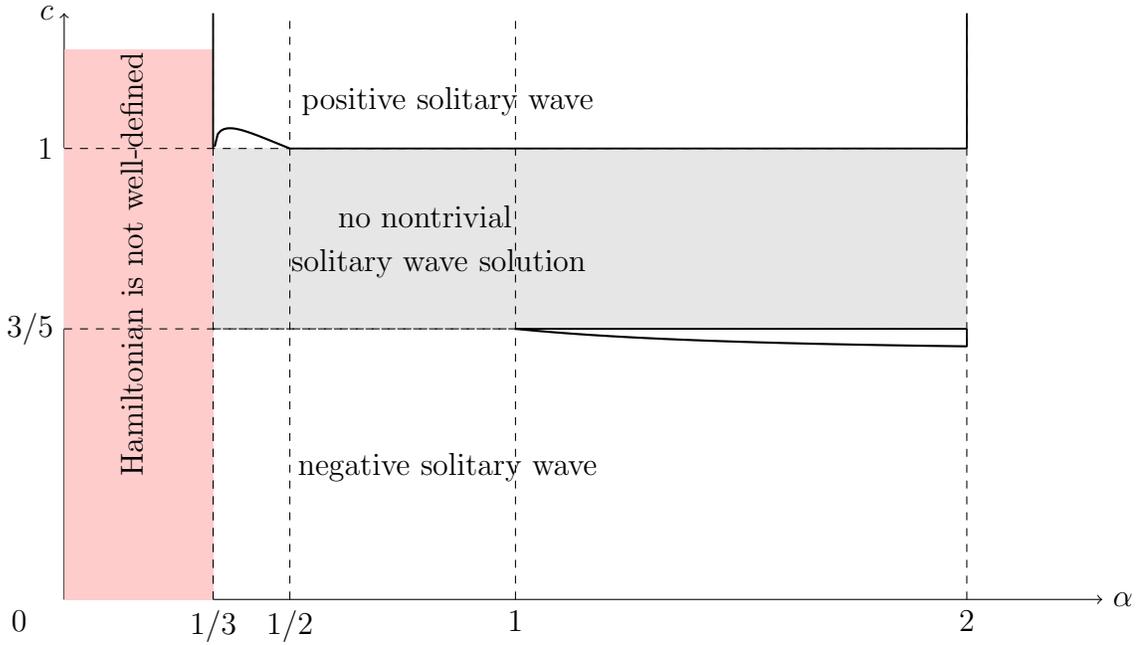

\indent For the case $p=2$ and $\alpha\in\left(\frac{1}{2},1\right]$, we have a different scenario. For the even values of $p$ there are no negative wave solutions therefore we are considering only the case $c>1$ where there exist positive solitary waves.  In this case
\begin{itemize}
  \item For $\frac{1}{3}<\alpha\leq \frac{1}{2}$  and for $1<\alpha\leq 2$ we have that  $\displaystyle \frac{d}{dc}\mathcal{K}(c)>0$ for all $c\in (1,+\infty)$
  \item For $\frac{1}{2}<\alpha\leq 1$ we have that  $\displaystyle \frac{d}{dc}\mathcal{K}(c)>0$ for all $(c_1,+\infty)$ and  $\displaystyle \frac{d}{dc}\mathcal{K}(c)<0$ for all $c\in(1,c_1)$ where
      $$c_1=\frac{3\alpha+1+\sqrt{2\alpha-1}}{5\alpha}.$$
\end{itemize}
The variations of the roots of $\displaystyle \frac{d}{dc}\mathcal{K}(c)$  with $\alpha$ are presented in the right  panel of Figure 2.
 This analysis together with Proposition $\ref{theolinop}$ enable us to conclude that the solitary wave $Q_c$ is spectrally (orbitally) stable,
for all $c>1$ when $\alpha\in (\frac{1}{3}, ~\frac{1}{2} ] \cup (1,2]$ and for all $c>c_1$ when $\alpha\in (\frac{1}{2}, 1 )$. The solitary wave $Q_c$ is spectrally unstable if $1<c<c_1$  for the case  $\alpha\in (\frac{1}{2}, 1 )$.\\
In case of general values of $p$, the derivative of $\mathcal{K}(c)$ has the roots:
\begin{eqnarray*}
 c_1=\frac{6\alpha+2p+3\alpha p +\sqrt{2}p\sqrt{2 \alpha-p + \alpha p} }{5 \alpha(p+2)},\hskip1cm
  c_2=\frac{6\alpha+2p+3\alpha p -\sqrt{2}p\sqrt{2 \alpha-p + \alpha p} }{5 \alpha(p+2)}
\end{eqnarray*}
where $c_1>c_2$ for all $\alpha$ and $p$. The above stability analysis for the cases $p=1$ and $p=2$ can be generalized for even and odd values of $p$.


\setcounter{equation}{0}
\section{Numerical Results}
In this section, we first construct the solitary wave solutions by using Petviashvili's iteration  method.
The Petviashvili method  for the gfBBM equation is given by
\begin{equation}
   \widehat{Q}_{n+1}(k)=\frac{(M_n)^{\nu}}{2\left[\left(\displaystyle\frac{5c}{4}-\frac{3}{4}\right)|k|^{\alpha} +c-1\right]}\widehat{Q_{n}^{p+1}}(k) \label{scheme}
\end{equation}
with stabilizing factor
\begin{equation*}
  M_n=\displaystyle\frac{\displaystyle\int_{\mathbb{R}} \left[\left(\displaystyle\frac{5c}{4}-\frac{3}{4}\right)|k|^{\alpha} +c-1\right] [\widehat{Q}_{n}(k)]^2 dk }{\frac{1}{2}\int_{\mathbb{R}}\displaystyle \widehat{Q}^{p+1}_{n}(k) \widehat{Q}_{n}(k)dk },
\end{equation*}
for the parameter $\nu=({p+1})/{p}$. Here $\widehat{Q}$ denotes the Fourier transform of $Q_c$. Then we investigate time evolution of the  solutions by using a numerical scheme combining a Fourier pseudo-spectral method for space and a fourth order Runge-Kutta method for the time integration.  We assume that $u(x,t)$ has periodic boundary condition $u(-L,t)=u(L,t)$ on the truncated domain $(x,t) \in [-L, L] \times [0,T]$.
In the following numerical experiments the space interval and the number of spatial grid points are chosen as $[-4096, 4096]$ and $N=2^{18}$, respectively. The time step is $\Delta t=5\times 10^{-4}$.
We refer \cite{oruc} for the details of the numerical methods.

\begin{figure}[htb!]
 \begin{minipage}[t]{0.45\linewidth}
   \includegraphics[width=3in]{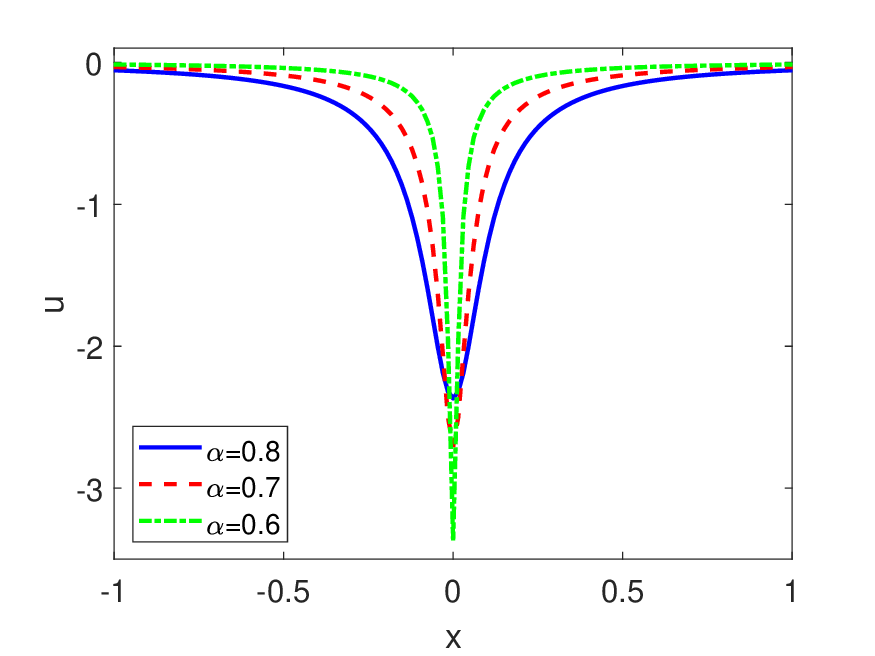}
 \end{minipage}
 \hspace{30pt}
\begin{minipage}[t]{0.45\linewidth}
   \includegraphics[width=3in]{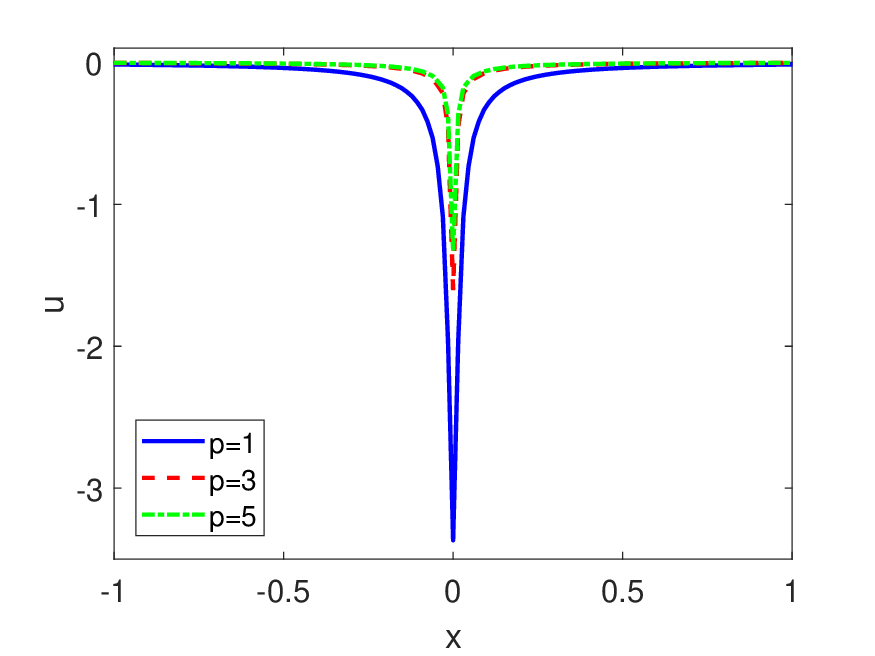}
 \end{minipage}
  \caption{Variation of negative solitary wave profiles with $\alpha$ (left panel) and with $p$ (right panel) when $c=0.5$. }
\label{fig:Nsw}
\end{figure}

We present the negative solitary wave profiles for various values of $\alpha$ when $p$ is fixed, and for various values of $p$ when $\alpha$ is fixed in Figure \ref{fig:Nsw}. In both cases we choose $c=0.5$. In the case of positive solitary waves it is observed in \cite{oruc} that decreasing the order of fractional derivative and increasing the order of nonlinearity have the same effect on the solutions. However, we do not observe the same behaviour for the negative solitary waves. The amplitude of the negative solitary wave solution increases with the decreasing values of $\alpha$ whereas the amplitude decreases with the increasing values of $p$.

In order to study the stability of solitary wave solutions we choose the initial condition as a perturbation of the solitary wave solution $Q_c$ obtained by Petviashvili method. Hence the initial condition is of the form
\begin{equation}\label{initialcond}
  u_0=\gamma Q_c(x)
\end{equation}
where $\gamma$ is the perturbation parameter. In the following numerical experiments we choose $\gamma=1.1$. As a control of the numerical accuracy we make sure that the  error $|F(t)-F(0)|$ in the conserved quantity is less than $10^{-5}$. We set $p=1$  for the rest of the numerical experiments. 

\begin{figure}[htb!]
 \begin{minipage}[t]{0.45\linewidth}
   \includegraphics[width=3in]{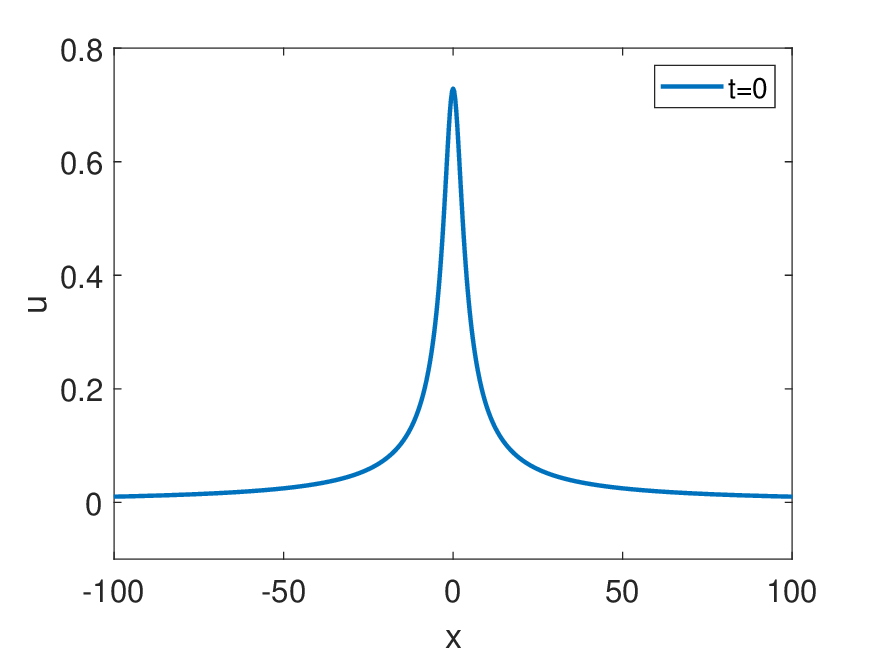}
 \end{minipage}
 \hspace{30pt}
\begin{minipage}[t]{0.45\linewidth}
   \includegraphics[width=3in]{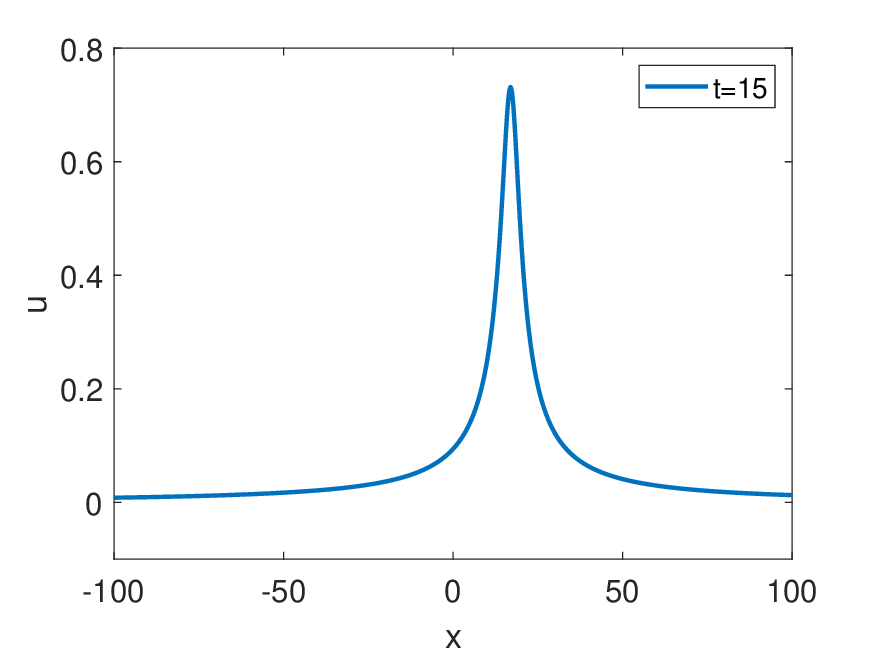}
 \end{minipage}
  \caption{Evolution of a perturbed orbitally stable positive solitary wave with, $\alpha=0.6$ and $c=1.1$.}
 \label{fig:stable1}
\end{figure}

In Figure \ref{fig:stable1}, we illustrate propagation of the perturbed positive solitary wave solution for $\alpha=0.6$ and $c=1.1$.  Analytical results indicate an orbital stability for these values.  The numerical result indicates a nonlinear  stability which is compatible with theoretical result.

\begin{figure}[htb!]
 \begin{minipage}[t]{0.45\linewidth}
   \includegraphics[width=3in]{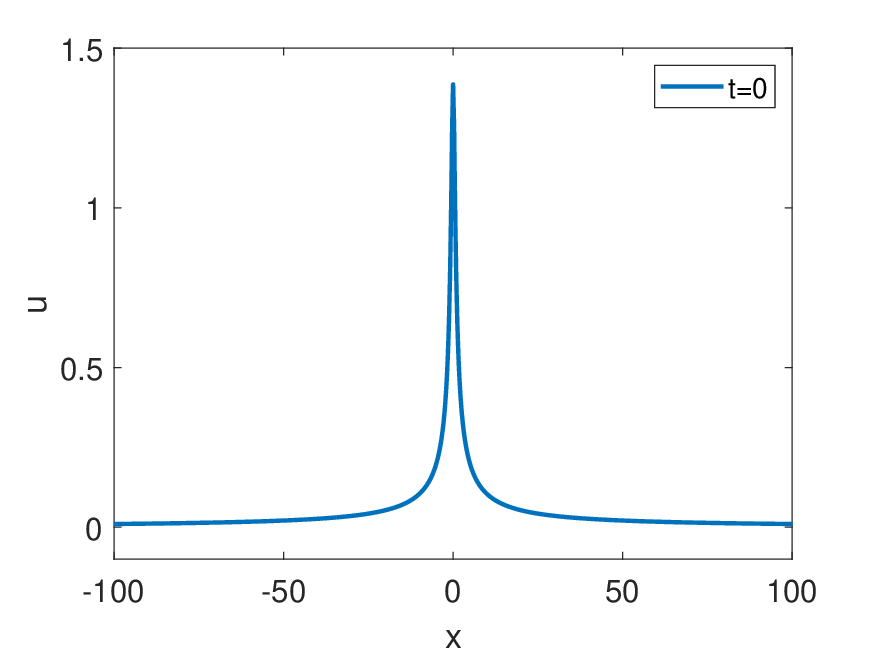}
 \end{minipage}
 \hspace{30pt}
\begin{minipage}[t]{0.45\linewidth}
   \includegraphics[width=3in]{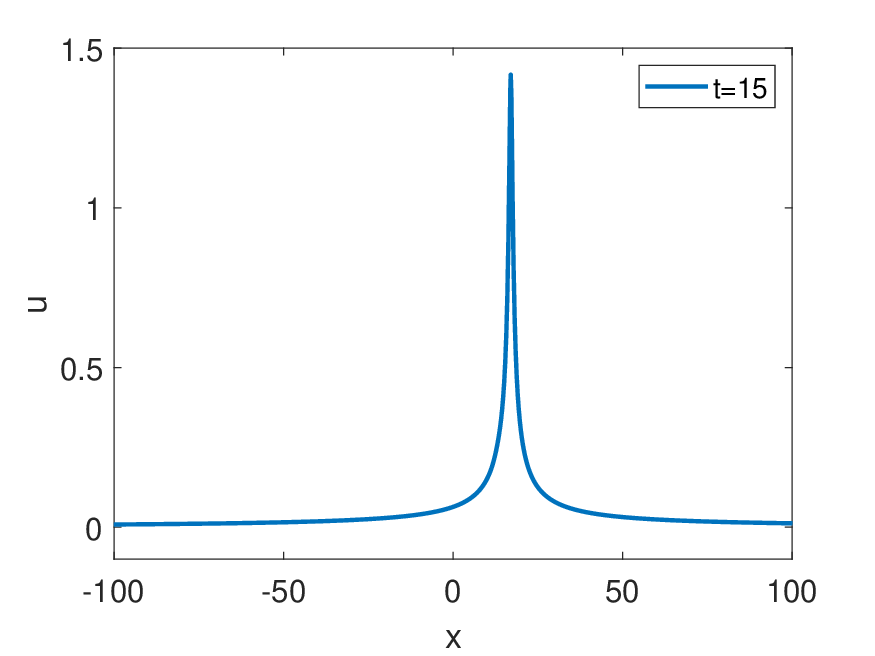}
 \end{minipage}
  \caption{Evolution of  a perturbed orbitally stable positive solitary wave with $\alpha=0.45$, $p=1$ and $c=1.1>c_1$.}
 \label{fig:stableb}
\end{figure}

Now, we focus on the interval $\frac{1}{3} < \alpha < \frac{1}{2}$ and $c>1$.
Here we have two subregions: For $c>c_1$, the solitary waves are orbitally stable whereas, for $1<c<c_1$ solitary waves are spectrally unstable. For $\alpha=0.45$ the critical wave speed is $c_1 \approx 1.04 $.
The evolution of the perturbed stable solution for $c=1.1>c_1$ is illustrated in Figure \ref{fig:stableb}.
Here we observe that the numerical result  agrees with the analytical  result of the stability.

\begin{figure}[htb!]
 \begin{minipage}[t]{0.45\linewidth}
   \includegraphics[width=3in]{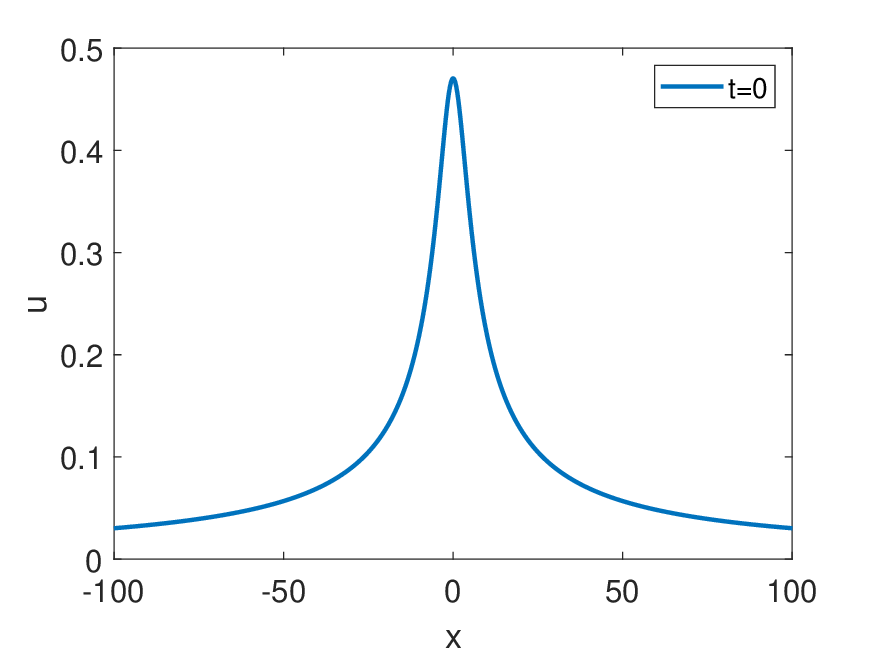}
 \end{minipage}
 \hspace{30pt}
\begin{minipage}[t]{0.45\linewidth}
   \includegraphics[width=3in]{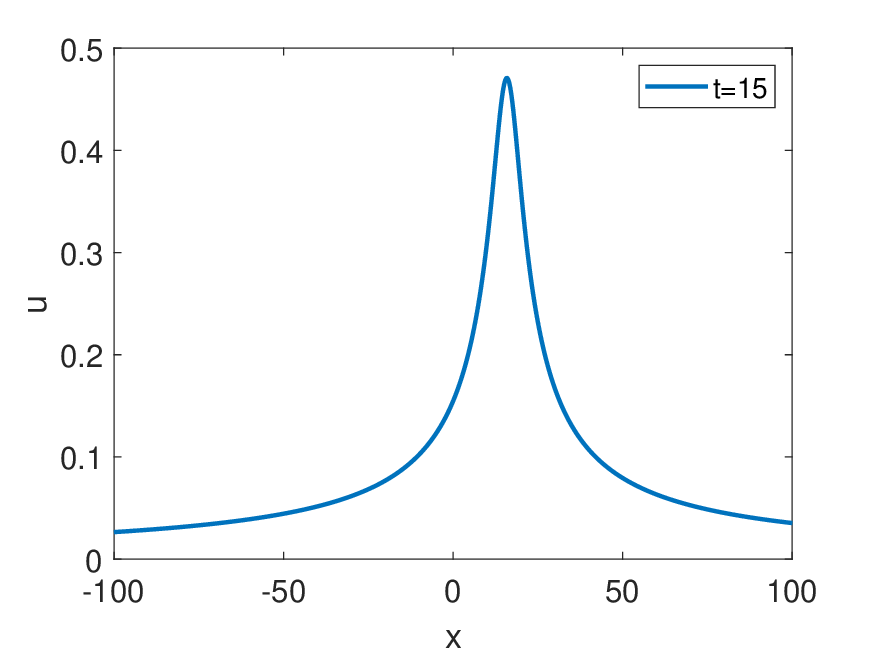}
 \end{minipage}
  \caption{Evolution of a spectrally unstable  perturbed positive solitary wave with $\alpha=0.45$, $p=1$ and $c=1.035<c_1$.}
 \label{fig:stablec}
\end{figure}


In Figure \ref{fig:stablec}, we investigate the time evolution of the perturbed solution for $\alpha=0.45$, $c=1.035<c_1$.  However, we could not observe a visible growth in the  solutions as time increases. The analytical results give the spectral instability but we study the nonlinear stability numerically.
One of the conditions to obtain that spectral instability implies orbital instability is that the mapping  $u_0\mapsto u(t)$ is smooth. For $\alpha\in (\frac{1}{3},\frac{1}{2})$, this is an open problem.
Therefore,
this observation does not imply a contradiction between the analytical and the numerical results.

\begin{figure}[htb!]
 \begin{minipage}[t]{0.45\linewidth}
   \includegraphics[width=3in]{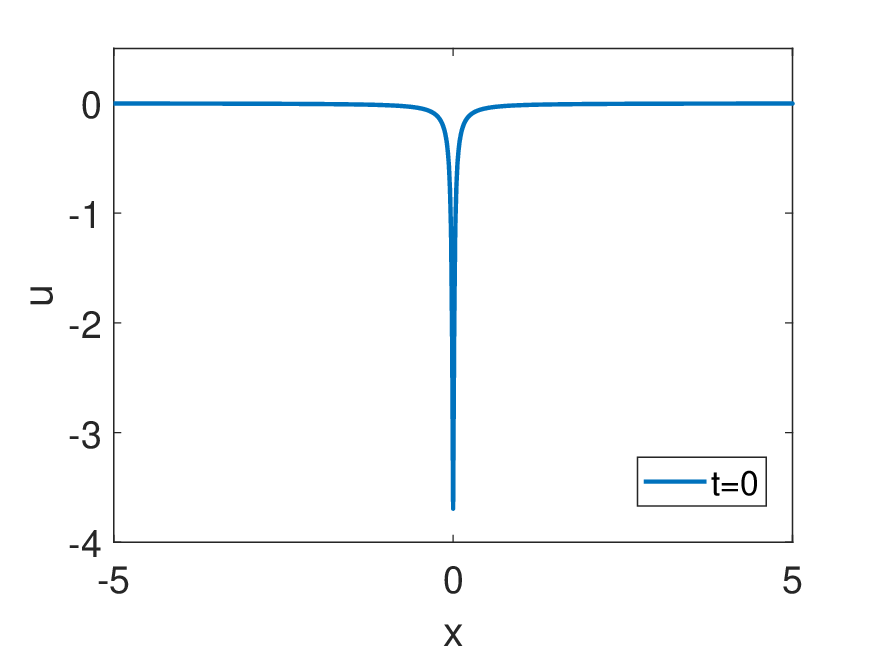}
 \end{minipage}
\hspace{30pt}
\begin{minipage}[t]{0.45\linewidth}
   \includegraphics[width=3in]{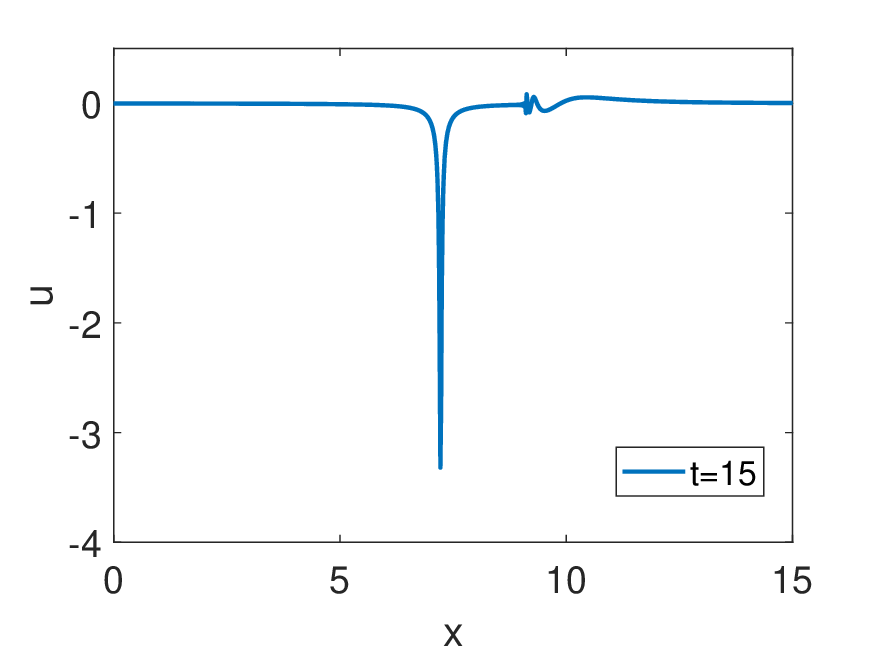}
 \end{minipage}
  \caption{Evolution of a spectrally unstable perturbed negative solitary wave solution with $\alpha=0.6$, $p=1$ and $c=0.5$}
 \label{fig:stablef}
\end{figure}

\indent The gfBBM equation has negative solitary wave solution when $c<\frac{3}{5}$ and $p=1$. The propagation of a perturbed negative solitary wave profile for $\alpha=0.6$ and $c=0.5$ is depicted in Figure \ref{fig:stablef}. The amplitude decreases slowly and  the small tail like oscillations appear by the time. As in the previous experiment, the numerical result does not imply a nonlinear instability.



\begin{thebibliography}{10}

\bibitem{abdelouhab}
L. Abdelouhab, J.L. Bona, M. Felland, and J.C. Saut.
\newblock Nonlocal models for nonlinear, dispersive waves.
\newblock {\em Physica D}, 40(3):360--392, 1989.


\bibitem{adams}
R.A. Adams and J.F. Fournier.
\newblock {\em Sobolev spaces}.
\newblock Elsevier, 2003.


\bibitem{le}
S. Amaral, H. Borluk, F.Natali, G.M. Muslu, and G. Oruc.
\newblock On the existence, uniqueness, and stability of
periodic waves for the fractional
Benjamin–Bona–Mahony equation.
\newblock {Stud. Appl. Math.} http://doi.org/10.1111/sapm.12428.

\bibitem{angulo}
J. Angulo.
\newblock Nonlinear Dispersive Equations: Existence and Stability of Solitary and Solitary Travelling Waves.
\newblock {\em Mathematical Surveys and Monographs 156}.
\newblock American Mathematical Society, 2009.

\bibitem{benjamin2}
T.B. Benjamin.
\newblock Model equations for long waves in nonlinear dispersive systems
\newblock {Philos. Trans. Royal Soc. A Philos T.R. Soc. A.}, 272(1220):47-78, 1972.

\bibitem{bona}
J.L. Bona, P.E. Souganidis, and W.A. Strauss.
\newblock Stability and instability of solitary waves of Korteweg-de Vries type.
\newblock {\em Proc. Real Soc. Lond. A}, 411: 395--412, 1987.

\bibitem{dingemans}
M.W. Dingemans.
\newblock Water wave propagation over uneven bottoms: Nonlinear wave propagation.
\newblock Vol. 13. World Scientific, 1997.


\bibitem{duran}
A. Dur\'{a}n.
\newblock An efficient method to compute solitary wave solutions of fractional
  {K}orteweg–de {V}ries equations.
\newblock {\em Int. J. Comp. Math.}, 95(6-7):1362--1374, 2018.


\bibitem{EEE1}
H.A. Erbay, S. Erbay, and A. Erkip.
\newblock Derivation of {C}amassa-{H}olm equations for elastic waves.
\newblock {\em Phys. Lett. A}, 379:956--961, 2015.

\bibitem{EEE2}
H. A Erbay, S. Erbay, and A. Erkip.
\newblock Derivation of generalized {C}amassa-{H}olm equations from
  {B}oussinesq type equations.
\newblock {\em J. Non. Math. Phys.}, 23:314--322, 2016.



\bibitem{LG}
L.G. Farah, J. Holmer, and S. Roudenko.
\newblock Instability of solitons–revisited, I: The critical generalized KdV equation.
\newblock {\em Contemp. Math.}, 729: 65-88, 2019.

\bibitem{levy}
R. Fetecau and  D. Levy.
\newblock Approximate model equations for water waves.
\newblock {\em Commun. Math. Sci.},  3(2):159-170, 2005.


\bibitem{franklenzmann}
R.L. Frank and E. Lenzmann.
\newblock Uniqueness of non-linear ground states for fractional Laplacians in
  $\mathbb{R}$.
\newblock {\em Acta Math.}, 210(2):261--318, 2013.


\bibitem{grillakis}
M. Grillakis, J. Shatah, and W. Strauss.
\newblock Stability theory of solitary waves in the presence of symmetry I.
\newblock {\em J. Funct. Anal.}, 74(1):160--197, 1987.


\bibitem{he}
J. He and Y Mammeri.
\newblock Remark on the well-posedness of weakly dispersive equations.
\newblock {\em ESAIM. Proceedings and Surveys}, 64:111--120, 2018.

\bibitem{kalisch}
M. Kalisch and N.T. Nguyen.
\newblock Stability of negative solitary waves.
\newblock {\em Elect. J. Diff. Equat.}, 158:1--20, 2009.


\bibitem{kalisch1}
M. Kalisch.
Solitary waves of depression.
\newblock {\em J. Comput. Anal. Appl.}, 8:5--24, 2006.


\bibitem{KP} T. Kapitula and K. Promislow.
\newblock Spectral and dynamical stability of nonlinear waves.
	\newblock {\em Appl. Math. Sci.} Springer, New York, 2013.


\bibitem{KS}
T. Kapitula and A. Stefanov.
\newblock Hamiltonian-{K}rein (instability) index theory for {K}d{V}-like
  eigenvalue problems.
\newblock {\em Stud. Appl. Math.}, 132:183--211, 2014.



\bibitem{klein}
C. Klein and J.C. Saut.
\newblock A numerical approach to blow-up issues for dispersive perturbations
  of {B}urgers equation.
\newblock {\em Physica D}, 295:46 -- 65, 2015.


\bibitem{lewis}
J. Courtenay Lewis, and J.A. Tjon.
\newblock Resonant production of solitons in the RLW equation.
\newblock {\em Phys. Lett. A}, 73:275--279, 1979.



\bibitem{linares1}
F. Linares, D. Pilod, and J.C. Saut.
\newblock Dispersive perturbations of {B}urgers and hyperbolic equations {I}:
  local theory.
\newblock {\em SIAM J. Math. Anal.}, 46(2):1505--1537, 2014.

\bibitem{linares2}
F. Linares, D. Pilod, and J.C. Saut.
\newblock Remarks on the orbital stability of ground state solutions of
  f{K}d{V} and related equations.
\newblock {\em Adv. Differ. Equ.}, 20(9-10):835--858, 2015.




\bibitem{martel}
Y. Martel and F. Merle.
\newblock Instability of solitons for the critical generalized Korteweg-de Vries equation.
\newblock {\em Geom. Funct. Anal.}, 11: 74--123, 2001.


\bibitem{kalisch3}
D. Moldabayeva, H. Kalisch, and D. Dutykh.
\newblock The Whitham Equation as a model for surface water waves.
\newblock {\em Physica D.},  309:99–107, 2015.

\bibitem{natalipastor}
F. Natali and A. Pastor.
\newblock The Fourth-Order Dispersive Nonlinear Schr\"odinger Equation: Orbital Stability of a Standing Wave.
\newblock {\em SIAM J. Appl. Dyn. Sys.}, 14: 1326-1347, 2015.

\bibitem{natali}
F. Natali, U. Le, and D. Pelinovsky.
\newblock New variational characterization of periodic waves in the fractional
  {K}orteweg-de {V}ries equation.
\newblock {\em Nonlinearity}, 33:1956--1986, 2020.


\bibitem{oruc}
G. Oruc, H. Borluk, and G.M. Muslu.
\newblock The generalized fractional Benjamin-Bona-Mahony equation: Analytical
  and numerical results.
\newblock {\em Physica D}, (132499), 2020.

\bibitem{pava}
J.A. Pava.
\newblock Stability properties of solitary waves for fractional {K}d{V} and
  {B}{B}{M} equations.
\newblock {\em Nonlinearity}, 31(3):920--956, 2018.

\bibitem{strauss}
P.E. Souganidis and W. Strauss.
\newblock Instability of a class of dispersive solitary waves.
\newblock {\em Proc. Royal Soc. Edinb.}, 114A: 195--212, 1990.


\bibitem{weinstein}
M. Weinstein.
\newblock Modulational stability of ground states of nonlinear Schrödinger equations.
\newblock {\em SIAM J. Math. Anal.}, 16: 472-491, 1985.













\end{thebibliography}
\end{document}